\documentclass[centertags, reqno]{amsart}               
\usepackage{graphicx}
\usepackage{url}    
\usepackage{dsfont}

\newcommand{\id}{\mathds{1}}

\newtheorem{thm}{Theorem}
\newtheorem*{thm*}{Theorem} 
\newtheorem{lemma}{Lemma}
\newtheorem{cor}{Corollary}
\newtheorem{prop}{Proposition}

\theoremstyle{definition}

\newtheorem{question}{Question}

\theoremstyle{remark}
\newtheorem{rem}{Remark} 
\newtheorem{ex}{Example}

\newcommand{\mr}{{\mathbb R}}

\newcommand{\mn}{{\mathbb N}}

\newcommand{\mc}{{\mathbb C}}

\renewcommand{\rho}{\varrho}
\newcommand{\eps}{\varepsilon}

\renewcommand{\Im}{\operatorname{Im}}
\renewcommand{\Re}{\operatorname{Re}}

 \newcommand{\Dom}{\operatorname{D}} 

\newcommand{\ran}{\operatorname{Ran}}
\newcommand{\hil}{\mathcal{H}}
\newcommand{\bdd}{\mathcal{B}}
\newcommand{\num}{\operatorname{Num}}
\newcommand{\cnum}{\overline{\operatorname{Num}}}
\newcommand{\dom}{\operatorname{D}}
\renewcommand{\ker}{\operatorname{Ker}}

\begin{document}
   
\title[Non-round points of the boundary of the numerical range]{On non-round points of the boundary of the numerical range and an application to non-selfadjoint Schr\"odinger operators}

\author[M. Hansmann]{Marcel Hansmann}
\address{Faculty of Mathematics\\ 
Chemnitz University of Technology\\
Chemnitz\\
Germany.}
\email{marcel.hansmann@mathematik.tu-chemnitz.de}

 \begin{abstract}
We show that non-round boundary points of the numerical range of an unbounded operator (i.e. points where the boundary has infinite curvature) are contained in the spectrum of the operator. Moreover, we show that  non-round boundary points, which are not corner points, lie in the essential spectrum. This generalizes results of H\"ubner, Farid, Spitkovsky and Salinas and Velasco for the case of bounded operators.

We apply our results to non-selfadjoint Schr\"odinger operators, showing that in this case the boundary of the numerical range can be non-round only at points where it hits the essential spectrum.
 \end{abstract}  

 
\maketitle 

\section{Introduction}
\emph{When is a boundary point of the numerical range of a linear operator contained in the spectrum of the operator?} This is the main question this paper is concerned with. Let us recall that the numerical range of a linear operator $A$ in the complex Hilbert space $(\hil, \langle . , . \rangle)$ is defined as 
\begin{equation}
  \label{eq:14}
  \num(A) = \{ \langle Af,f \rangle : f \in \dom(A), \|f\|=1\}
\end{equation}
and that by the Toeplitz-Hausdorff theorem it is always a convex set. In the finite dimensional case the numerical range is also compact, but in the above generality it needs neither be bounded nor closed. In most circumstances (more about that below) the spectrum of $A$ is contained in the closure of the numerical range. 

The above question already has a long history and goes back, in the infinite-dimensional case, to a paper of Donoghue \cite{MR0096127}, who dealt with bounded operators and  showed that corner points of the boundary of the  numerical range, if they are elements of the numerical range, are eigenvalues of the operator. For corner points which are not elements of the numerical range, Hildebrandt \cite{MR0200725} (and also Sims \cite{MR0338802}) showed that they are contained in the approximate point spectrum. Some decades later, H\"ubner \cite{MR1371343} generalized Hildebrandt's result to points where the boundary of the numerical range is non-round (i.e. where the boundary has infinite curvature). Even more recently, H\"ubner's results where refined, somewhat independently, by Farid \cite{MR1731863}, Spitkovsky 
\cite{MR1752163} and Salinas and Velasco \cite{MR1800238}. Among other things, they showed that non-round points of the boundary of the numerical range, which are not corner points, are elements of the essential spectrum. 

All of the above results have one feature in common: they were proved for bounded operators only. It is the aim of this paper to extend these results, as far as possible, to the case of unbounded operators. As we will see, this is not a straightforward task since most of the proofs of the above results (only the proof of Donoghue's theorem being an exception) rely heavily on the boundedness of the operators and on the fact that they are everywhere defined. 
 
In the second part of this paper, we will apply the results of the first part to the study of non-selfadjoint Schr\"odinger operators $-\Delta + V$ in $L^2(\mr^d)$, with a complex valued potential $V$. While these operators have received a lot of attention in recent times, see e.g. \cite{MR1819914, MR2260376, MR2540070, MR2559715, MR2596049, MR2651940, FLS11, Frank11, Hansmann11, MR3016473, MR3077277}, as compared to their selfadjoint counterparts they are still far from being well understood. For instance, while it is not too difficult to obtain bounds on the numerical range of these operators (see \cite{b_Davies}), it is usually quite hard, or impossible, to determine the numerical range exactly. We will shed a little more light on the structure of the numerical range of $-\Delta + V$ by showing that under very mild assumptions on the imaginary part of the potential, the boundary of the numerical range can be non-round only at points where it hits the essential spectrum. In particular, if the spectrum is purely discrete, then the boundary of the numerical range will consist entirely of round points. 
    
The plan for this paper is as follows: In Section 2, we gather some preliminary material about convex sets and their boundary points. In Section 3 and 4 we will state and prove our general results about boundary points of the numerical range. Finally, the non-selfadjoint Schr\"odinger operator will be considered in Section 5. 

\section{Preliminaries about convex sets and their boundary points}
 
Let $\Omega \subset \mc$ be a closed, convex set with an interior point  
and let $\lambda \in \partial \Omega$. Then there exists at least one supporting line $l_\lambda$ for $\Omega$ passing through $\lambda$. If there exists more than one such line, then $\lambda$ is called a \emph{corner point} of $\partial \Omega$ and $\Omega$ is contained in a closed sector  with vertex $\lambda$ and semivertical angle less than $\pi/2$. While there always exists more than one such sector, we will simply pick the smallest sector with these properties  and choose the supporting line $\l_\lambda$ which is orthogonal to the axis of this sector. Having fixed the supporting line, we now choose a rectangular system of coordinates $(\xi, \eta)$ with an origin at $\lambda$, the $\xi$-axis coinciding with $l_\lambda$ and directed such that $\Omega \subset \{ (\xi ,\eta) : \eta \geq 0\}$. In the following, when using coordinates it will always refer to this coordinate system. 

Let $D_\eps':=\{ (\xi,\eta) : \xi^2 + \eta^2 \leq \eps^2, \xi \neq 0\}$ and  
note that $\partial \Omega \cap D_\eps' \neq \emptyset$ for all $\eps > 0$. We define the \emph{right-hand upper curvature} of $\partial \Omega$ at $\lambda$ as
\begin{equation}
  \label{eq:201}
\gamma_u^+(\lambda) :=  \limsup_{\tiny \begin{matrix}
\eta \to 0, \xi \downarrow 0\\
(\xi,\eta) \in \partial\Omega
\end{matrix}}
 \frac \eta {\xi^2} := \lim_{\eps \downarrow 0} \sup \left\{ \frac{\eta}{\xi^2}: (\xi,\eta) \in \partial \Omega \cap D_\eps', \: \xi > 0 \right\}.
\end{equation}
The \emph{right-hand lower curvature} of $\partial \Omega$ at $\lambda$ is defined as 
\begin{equation}
  \label{eq:11}
  \gamma_l^+(\lambda) :=  \liminf_{\tiny \begin{matrix}
\eta \to 0, \xi \downarrow 0 \\
(\xi,\eta) \in \partial\Omega
\end{matrix}}
 \frac \eta {\xi^2}.
\end{equation}
If $\gamma_l^+(\lambda)=\gamma_u^+(\lambda)$, then the joint value is called the \emph{right-hand curvature}, $\gamma^+(\lambda)$, of $\partial \Omega$ at $\lambda$. The left-hand (upper/lower) curvatures $\gamma_u^-(\lambda), \gamma_l^-(\lambda), \gamma^-(\lambda)$ are defined analogously.              Moreover, the \emph{upper and lower curvatures}, $\gamma_u(\lambda)$ and $\gamma_l(\lambda)$, of $\partial \Omega$ at $\lambda$ are defined as in (\ref{eq:201}) and (\ref{eq:11}), but with the right-limit $\xi \downarrow 0$ replaced by the ordinary limit $\xi \to 0$. Equivalently, $\gamma_l(\lambda)=\min( \gamma_l^+(\lambda), \gamma_l^-(\lambda))$ and $\gamma_u(\lambda)= \max( \gamma_u^+(\lambda), \gamma_u^-(\lambda))$. If  the lower- and upper curvatures of $\partial \Omega$ at $\lambda$ coincide, then the joint value is called the \emph{curvature} $\gamma(\lambda)$ of $\partial \Omega$ at $\lambda$. Note that our definition of curvature differs from the usual one by a factor of $1/2$. However, as we will mainly be interested in points where the curvature is infinite, this should not lead to much confusion.

We call $\lambda$ a \emph{point of infinite curvature}, if $\gamma(\lambda)=\infty$ (which is the case if and only if $\gamma_l(\lambda)=\infty$) and a \emph{point of infinite upper curvature}, if $\gamma_u(\lambda)=\infty$. With our choice of supporting line, if $\lambda$ is a corner point then it is a point of infinite curvature. Finally, we say that $\lambda$ is of \emph{unilateral infinite curvature} if the right- or left-hand curvatures, or both, of $\partial \Omega$ at $\lambda$ are infinite. Note that if $\lambda$ is of unilateral infinite curvature, then it is of infinite upper curvature. We will see below (Example \ref{example1}.(iii)) that the converse need not be true.
\begin{rem}
 In H\"ubner's paper \cite{MR1371343} points of infinite curvature where called \emph{non-round points} and for reasons of brevity we borrowed this term for the title and introduction of this paper. In the following, however, we will not use this term again and speak about curvature instead. 
\end{rem}

\begin{ex}\label{example1} In the following examples we choose $\Omega \subset \mr^2$ as the epigraph of a convex function $f : [-1,1] \to \mr_+$. 
 \begin{enumerate} 
 \item[(i)] $f(\xi)=|\xi|^\alpha$:  Then $\gamma(0)= 0$ if $\alpha > 2$ and $\gamma(0)=1$ if $\alpha=2$. If $1<\alpha<2$ then $0$ is a point of infinite curvature and if $\alpha = 1$ it is a corner point.
 \item[(ii)] $f(\xi)=(-\xi)^{3/2}$ for $\xi < 0$ and $f(\xi)=\xi^2$ for $\xi \geq 0$: Here $\gamma^+(0)=\gamma_l(0)=1$ and $\gamma^-(0)=\gamma_u(0)=\infty$.
 \item[(iii)] $\gamma_l(0)$ can be different from $\gamma_u(0)$ also in case that the function $f$ is even: Let $g:[0,1] \to [0,1]$ be a monotone increasing, polygonal curve which satisfies $x^2 \leq g(x) \leq \sqrt{x}$ for all $x \in [0,1]$ and such that $0$ is a limit point of both
$ \{ x : g(x) = \sqrt{x}\}$ and $\{ x : g(x) = x^2\}$. Let $f(\xi)=\int_0^\xi g(t) dt$ for $\xi \in [0,1]$ and $f(\xi):=f(-\xi)$ for $\xi \in [-1,0)$. Then $f \in C^1[-1,1]$ is convex, $\gamma_u(0) = \gamma^\pm_u(0)=  \infty$ and $\gamma_l(0) = \gamma_l^\pm (0)=0$.  
 \end{enumerate}
In view of the last example we should remark that in \cite{MR593634} it was actually shown that for most convex bodies and for most of their boundary points $\lambda$ (in each case meaning all except those in a set of first Baire category) one has $\gamma_u(\lambda)=\infty$ and $\gamma_l(\lambda)=0$.
\end{ex}

The following equivalence is probably well known among experts. 
\begin{lemma}\label{lem:upper}
Let $\Omega \subset \mc$ be a closed, convex set with an interior point and let $\lambda \in \partial \Omega$.  Then the following are equivalent:
\begin{enumerate}
\item $\lambda$ is a point of infinite upper curvature.
\item There does not exist a closed, non-degenerate disk $D$ such that $\lambda \in D \subset \Omega$. 
\end{enumerate}
\end{lemma}
\begin{proof}
Note that $\lambda$ is a point of finite upper curvature if and only if $\Omega$ contains a piece of parabola touching the boundary at $\lambda$. The domain bounded by this parabola contains a small non-degenerate disk $D \subset \Omega$ such that $\lambda \in \partial D$.

Concerning the other direction we note that if there exists a closed, non-degenerate disk $D$ with $\lambda \in D \subset \Omega$, then this disc contains a piece of parabola with the above properties.
\end{proof}

\section{Main results}   

Let $A$ be a densely defined, closed operator in the complex Hilbert space $(\hil, \langle ., . \rangle)$. As remarked above, by the Toeplitz-Hausdorff theorem its numerical range $\num(A)$, and hence its closure $\cnum(A)$, is a convex set. In the following, without further mentioning, we assume that $\num(A)$ does contain an interior point and that $\num(A) \neq \mc$. 
\begin{rem}
If $\num(A)$ does not have interior points then it is an interval. In this case there exist $\alpha, \beta \in \mc$ such that the operator $\alpha A + \beta \id$  is symmetric, and all known spectral results for symmetric operators easily translate into corresponding results for $A$.
\end{rem}
 
\begin{thm}\label{thm:ev}
Let $\lambda \in \partial \num(A)$ be a point of infinite upper curvature. If $\lambda \in \num(A)$, then $\lambda$ is an eigenvalue of $A$.
\end{thm}
A version of this theorem has first been established by Donoghue \cite{MR0096127}, who considered the case of corner points of bounded operators (in \cite{MR1839845} this was extended to corner points of the quadratic numerical range).  While we haven't found the above generalization to unbounded operators and to points of infinite upper curvature in the literature, the result might be known to the experts in the field as its proof doesn't require much changes as compared to Donoghue's original result. Our proof follows along the lines of the proof of Donoghue's theorem given in \cite{Shapiro}.      
\begin{proof}
Let $f \in \dom(A)$ with $\|f\|=1$ such that $\langle Af,f \rangle=\lambda$. If we can show that $f$ is an eigenfunction of $f$, then necessarily $Af=\lambda f$ and we are done. 

Let us suppose that $f$ is not an eigenfunction of $A$, i.e. $f$ and $Af$ are linearly independent, and derive a contradiction. To this end, let $P$ denote the orthogonal projection onto the two-dimensional Hilbert space $\hil_0=\operatorname{span}\{f,Af\}$ and let $A_0=PAP$, acting on $\hil_0$. Then $\num(A_0)$ is an ellipse (possibly degenerated to a line segment or even to a point) whose foci are the eigenvalues of $A_0$, see \cite{MR1417493} Section 1.1.   Since $\lambda$ is a point of infinite upper curvature of $\partial \num(A)$, Lemma \ref{lem:upper} implies that there does not exist a non-degenerate ellipse $E$ such that  $\lambda \in E \subset \cnum(A)$. But $\lambda \in \num(A_0) \subset \num(A)$, so $\num(A_0)$ must be a proper line segment or a single point. If it is a single point, then $\num(A_0)=\{\lambda\}$ and $A_0=\lambda \id$, which implies that  $Af = A_0 f = \lambda f$ and leads to a contradiction. On the other hand, if $\num(A_0)$ is a proper line segment,  then, since $\lambda$ is an extreme point of $\overline{\num}(A)$ (as a point of infinite upper curvature), it must be one of the endpoints of this line segment and hence is an eigenvalue of $A_0$. If $\mu \neq \lambda$ denotes the other eigenvalue of $A_0$, then the corresponding normalized eigenfunctions $f_\lambda$ and $f_\mu$ are orthogonal, as follows from the fact that there exist $\alpha, \beta \in \mc, \alpha \neq 0$ such that $\alpha A_0 + \beta \id$ is symmetric. Since  $\hil_0=\operatorname{span}\{f_\lambda, f_\mu\}$ we can write $ f = \gamma f_\lambda + \delta f_\mu$, where $|\gamma|^2 + |\delta|^2 = \|f\|^2=1$, and $Af=A_0f = \gamma \lambda f_\lambda + \delta \mu f_\mu$. But this shows that $\lambda = \langle Af,f \rangle = |\gamma|^2 \lambda + |\delta|^2 \mu$, which implies that $\delta = 0$ and so $Af=\lambda f$, again leading to a contradiction. 
\end{proof}
Theorem \ref{thm:ev} immediately leads to the question what one can expect if the assumption that  $\lambda$ is an element of the numerical range is removed. 
\begin{thm} \label{thm:main}
Suppose that $\dom(A) \subset \dom(A^*)$. If  $\lambda \in \partial \num(A)$ is a point of unilateral infinite curvature, then $\lambda \in \sigma(A)$. 
\end{thm}
H\"ubner \cite{MR1371343} proved this theorem for bounded operators and points of infinite curvature. Still considering the case of bounded operators, Salinas and Velasco \cite{MR1800238} generalized it to points of unilateral infinite curvature. 
\begin{question}
Does the conclusion of Theorem \ref{thm:main} remain valid if $\lambda$ is only a boundary point of infinite \emph{upper} curvature? (The answer to this question seems to be unknown even in the bounded case.) 
\end{question}
\begin{question}
As we will see below, the assumption that $\dom(A) \subset \dom(A^*)$ will enter our proof of Theorem \ref{thm:main} rather naturally. However, is it really necessary?
\end{question}  
The proof of Theorem \ref{thm:main}, which requires considerable changes as compared to the proof of the bounded case,  will be given in the next section. First, let us discuss some of the consequences of this theorem: Recall that the approximate point spectrum of $A$ is defined as 
\begin{equation}
  \label{eq:1}
  \sigma_{ap}(A)=\{ \lambda \in \mc : \exists (u_n) \subset \dom(A), \|u_n\|=1, (A-\lambda \id)u_n \to 0 \}.
\end{equation}
It is well known that the topological boundary of $\sigma(A)$ is contained in $\sigma_{ap}(A)$. 
\begin{cor}\label{cor:1}
Suppose that $\dom(A) \subset \dom(A^*)$ and that $\sigma(A) \subset \cnum(A)$. If $\lambda \in \partial \num(A)$ is a point of unilateral infinite curvature, then  $\lambda \in \sigma_{ap}(A)$. 
\end{cor}
The assumption $\sigma(A) \subset \overline{\num}(A)$ is satisfied whenever $\mc \setminus \overline{\num}(A)$ is connected and 
contains a point which is not in the spectrum of $A$, see \cite{kato}. In particular, in the bounded case it is always satisfied.
\begin{proof}[Proof of Corollary \ref{cor:1}]
  From Theorem \ref{thm:main} we know that $\lambda$ is in $\sigma(A) \cap \partial \num(A)$. Since $\sigma(A) \subset \cnum(A)$ by assumption, it follows that $\lambda \in \partial \sigma(A) \subset \sigma_{ap}(A)$. 
\end{proof}
We  recall that the essential spectrum of $A$ is defined as
\begin{equation}
  \label{eq:10}
\sigma_{ess}(A) = \{ \lambda \in \mc : A-\lambda \id \text{ is not a Fredholm operator } \}  
\end{equation}
and that a linear operator in $\hil$ is called a Fredholm operator if it has closed range and its kernel and cokernel are finite dimensional. 
\begin{thm} \label{thm:ess}
Suppose that $\dom(A) \subset \dom(A^*)$ and that $\sigma(A) \subset \cnum(A)$.  If $\lambda \in \partial \num(A)$ is a point of unilateral infinite curvature, but not a corner point, then $\ran(A-\lambda \id)$ is not closed. In particular, $\lambda \in \sigma_{ess}(A)$.
\end{thm}
In the bounded case this theorem has been proved independently by Farid \cite{MR1731863}, Spitkovsky 
\cite{MR1752163} and Salinas and Velasco \cite{MR1800238}. Our proof will follow along the lines of Spitkovsky's proof, but, once again, it will require some adaptions to work for the unbounded case. We will need the following two lemmas. 
\begin{lemma}\label{lem:ker}
Let $\dom(A) \subset \dom(A^*)$ and suppose that $\lambda \in \partial \num(A)$. Then
\begin{equation}
  \label{eq:22}
\ker(A-\lambda \id) \subset \ker(A^*-\overline{\lambda}\id).  
\end{equation}
\end{lemma}
\begin{proof}
  See  \cite{MR3162253}, Theorem 1.
\end{proof}
Note that in case that $A \in \bdd(\hil)$, i.e. $A$ is a bounded operator with $\dom(A)=\hil$, one has equality in (\ref{eq:22}), see the paper of Spitkovsky \cite{MR1752163}.  
\begin{lemma}\label{lem:2}
  Let $\lambda \in \sigma_{ap}(A)$ and let $A-\lambda \id$ be injective. Then $\ran(A-\lambda \id)$ is not closed.
\end{lemma}
\begin{proof}
Let us assume that $\ran(A-\lambda \id)$ is closed. Then by the closed graph theorem the closed operator $(A-\lambda \id)^{-1} : \ran(A-\lambda \id ) \to \hil$ would be bounded. However, since $\lambda \in \sigma_{ap}(A)$ there exists $(u_n) \subset \dom(A), \|u_n\|=1,$ with $v_n:=(A-\lambda \id)u_n \to 0$. 
But then $w_n:= v_n/\|v_n\| \in \ran(A-\lambda \id), \|w_n\|=1$ and $\|(A-\lambda \id)^{-1} w_n\| \to \infty$, which leads to a contradiction.
\end{proof}

\begin{proof}[Proof of Theorem \ref{thm:ess}]
Using a suitable translation of the operator it is no restriction to consider the case $\lambda = 0$ only.
Let us assume that $\ran(A)$ is closed and derive a contradiction.

 First, since $0 \in \sigma_{ap}(A)$ by Corollary \ref{cor:1}, Lemma \ref{lem:2} shows that $A$ cannot be injective, so $\dim(\ker(A))>0$. Since $A$ is closed, its kernel is closed and we have 
$$ \hil = \ker(A) \oplus \ker(A)^\perp.$$
Let us set $\hil_1=\ker(A)$ and $\hil_2=\ker(A)^\perp$, which are both Hilbert spaces with the induced scalar product. Clearly, $\hil_1$ is an invariant subspace for $A$, $A|_{\hil_1}=0$ and $\num(A|_{\hil_1})=\{0\}$. Moreover, by Lemma \ref{lem:ker} we have $\ker(A) \subset \ker(A^*)$ and so  $$\ran(A)=\overline{\ran}(A)=\ker(A^*)^\perp \subset \ker(A)^\perp,$$ which shows that $\hil_2$ is an invariant subspace for $A$ as well.  Let us set $B=A|_{\hil_2}$, which is an injective and closed operator in $\hil_2$. Moreover, $\ran(B)=\ran(A)$ is closed. Now let us note that 
$$ \overline{\num}(A)=\overline{\operatorname{conv}}(\num(A|_{\hil_1}),\num(A|_{\hil_2}))=\overline{\operatorname{conv}}(\{0\},\num(B)).$$
Since $0$ (as a point of unilateral infinite curvature) is an extreme point of $\overline{\num}(A)$, but not a corner point, it follows that $0 \in \overline{\num}(B)$ and so $\cnum(A)=\cnum(B)$. Hence $\lambda=0$ is a point of unilateral infinite curvature of $\partial \num(B)$ as well and Corollary \ref{cor:1} implies that $0 \in \sigma_{ap}(B)$. But since $B$ is injective, Lemma \ref{lem:2} then implies that $\ran(B)$ is not closed, which leads to a contradiction.
\end{proof}

\section{Proof of theorem \ref{thm:main}}

We assumed that $\lambda$ is a point of unilateral infinite curvature of $\partial \num(A)$.  Using an affine  transformation of $A$, it is no restriction to assume that $\lambda=0$, that $\mr$ is a supporting line for $\cnum(A)$, that $\cnum(A) \subset \mc_+=\{ z : \Im(z)\geq 0 \}$ and that in case of a corner point the imaginary axis coincides with the axis of the smallest sector containing $\cnum(A)$. Moreover, it is no restriction to assume that
$0$ is a point of right-hand infinite curvature for $\partial \num(A)$.

Since we assumed that $\num(A)$ contains an interior point, there exists an interior point $\alpha \in \num(A)$ with
\begin{equation}
  \label{eq:24}
  0< \Re(\alpha)< 1, \quad 0 < \Im(\alpha) < 1 \quad \text{and} \quad |\alpha|<1
\end{equation}
and such that the half open line segment $(0, \alpha]$ is contained in $\num(A)$. Now let $0< \eps_n < 1$ be any null sequence (i.e. $\eps_n \to 0$ for $n \to \infty$). Then we can find $(u_n) \subset \dom(A), \|u_n\|=1,$ such that 
\begin{equation}
  \label{eq:21}
  \langle Au_n, u_n \rangle = \eps_n \alpha. 
\end{equation}
\begin{rem}
  In the bounded case, H\"ubner et al. now introduce a sequence $(e_n) \subset \hil$, with $\|e_n\|=1$ and $\langle u_n, e_n \rangle =0$, by setting
$$ Au_n =: \eps_n \alpha u_n + x_n e_n$$
and, using the boundedness of $A$ extensively, show that $x_n=\langle Au_n, e_n \rangle \to 0$. In particular, this implies that $Au_n \to 0$ and so $0 \in \sigma(A)$. In the unbounded case we will not be able to show that $Au_n \to 0$ and have to go along a different route.
\end{rem}
\begin{question}
  Is it true that $Au_n \to 0$ for $n \to \infty$?
\end{question}
The main tool in our proof of Theorem \ref{thm:main} is the following new result. Recall that by assumption we have $\dom(A) \subset \dom(A^*)$.
\begin{prop}\label{prop:1}
Let $(u_n)$ be as defined in (\ref{eq:21}) and let $(f_n) \subset \dom(A)$ such that
\begin{eqnarray*}
    \max(\sup_n \|f_n\|,\sup_n \|Af_n\|, \sup_n \|A^*f_n\|) < \infty.
\end{eqnarray*}
Then 
\[ \lim_{n \to \infty} \left( |\langle f_n, Au_n \rangle| + |\langle f_n, A^*u_n \rangle| \right) = 0. \]
\end{prop}
\begin{rem}
  If $A \in \bdd(\hil)$, then we can choose $f_n=Au_n$ and $f_n=A^*u_n$, respectively, and the proposition implies that $Au_n \to 0$ and $A^*u_n \to 0$, recovering the known results mentioned above. 
\end{rem}
The proof of this proposition is rather lengthy and will be given below. First, let us show how the proposition can be used to prove Theorem \ref{thm:main}. 
\begin{proof}[Proof of Theorem \ref{thm:main}]
We want to prove that $0 \in \sigma(A)$. So let us assume that this is not the case, i.e. $A$ is boundedly invertible, and derive a contradiction. To this end, let us choose $f_n = A^{-1} u_n$, with $u_n$ as above. Then $(f_n) \subset \Dom(A), \|f_n\| \leq \|A^{-1}\|$ and $\|Af_n\| = 1$. Moreover, since $\dom(A) \subset \dom(A^*)$, the operator $A^*A^{-1}$ is defined on $\hil$. Since $A^{-1}$ is bounded and $A^*$ is closed, it is easy to see that also $A^*A^{-1}$ is closed and then the closed graph theorem implies that $A^*A^{-1} \in \bdd(\hil)$. In particular, this implies that
$ \|A^*f_n\| \leq \|A^*A^{-1}\|$. We can now apply Proposition \ref{prop:1} to conclude that 
$$ \lim_{n \to \infty} \left( |\langle f_n, Au_n \rangle| + |\langle f_, A^*u_n\rangle| \right) = 0.$$
But 
$$ \langle f_n, A^*u_n \rangle = \langle Af_n, u_n \rangle = \langle u_n, u_n \rangle = 1$$
for all $n$, which leads to the desired contradiction, showing that our assumption that $0 \notin \sigma(A)$ must have been wrong.  
  
\end{proof}
The proof of Proposition \ref{prop:1} requires a series of preparatory lemmas: First, let us introduce a sequence $(v_n) \subset \dom(A)$ (whose precise form will be chosen below) which satisfies
\begin{equation}
  \label{eq:8}
 \max(\sup_n \|v_n\|,\sup_n \|Av_n\|, \sup_n \|A^*v_n\|) \leq 1,
\end{equation} 
and 
\begin{equation}
  \label{eq:19}
  \sup_n |\Re( \langle Av_n,v_n \rangle)| \leq \Re(\alpha)/2,
\end{equation}
with $\alpha$ as given in (\ref{eq:24}). Moreover, let us define a sequence $(c_n) \subset \{-1,1\}$, depending on $(u_n)$ (as defined in (\ref{eq:21})) and $(v_n)$, as follows:
\begin{equation}
  \label{eq:13}
  c_n := \left\{ 
    \begin{array}{cl}
      1, & \text{if }  \Re( \langle Av_n, u_n \rangle + \langle Au_n,v_n \rangle) \geq 0 \\
     -1, & \text{if }  \Re( \langle Av_n, u_n \rangle + \langle Au_n,v_n \rangle) < 0. 
    \end{array}\right.
\end{equation}
Finally, let us set
\begin{equation}
  \label{eq:2}
  w_n:= u_n+\sqrt{\eps_n}c_nv_n \in \dom(A),
\end{equation}
where $0 < \eps_n < 1$ was defined above.  
\begin{lemma}\label{lem4}
For every $n \in \mn$
  \begin{equation}
    \label{eq:3}
     1-\sqrt{\eps_n} \leq \|w_n\| \leq 1 + \sqrt{\eps_n}.
  \end{equation}
\end{lemma}
\begin{proof}
Use the triangle inequality and the fact that $\|u_n\|=1$ and $\|c_nv_n\| \leq 1$.
\end{proof}
Since $\langle Au_n,u_n \rangle = \eps_n \alpha$ and $c_n^2=1$,  a short computation shows that
\begin{equation}
  \label{eq:26}
  \langle Aw_n, w_n \rangle = \eps_n \alpha + \sqrt{\eps_n} c_n \left( \langle A v_n, u_n \rangle +  \langle Au_n,v_n \rangle \right) + \eps_n \langle Av_n,v_n \rangle.  
\end{equation}
\begin{lemma}\label{lem5} 
For all $n \in \mn$ we have
\begin{equation}
  \label{eq:4}
  \left| \langle A w_n, w_n \rangle \right| \leq 4 \sqrt{\eps_n}. 
\end{equation}
\end{lemma}
\begin{proof}
This follows from (\ref{eq:26}) and the Cauchy-Schwarz inequality, using that 
$0< \eps_n < \sqrt{\eps_n} <1$, $|c_n|=1$, $\|u_n\| = 1, |\alpha| < 1$ and $\max(\|v_n\|,\|Av_n\|,\|A^*v_n\|) \leq 1$.
\end{proof}
\begin{lemma} \label{lem3}
For all $n \in \mn$ we have
\begin{equation}
  \label{eq:16}
   \Im( \langle A w_n, w_n \rangle) \leq 4 \eps_n
\end{equation}
and
\begin{equation}
  \label{eq:15} 
\left| \Im(\langle A v_n, u_n \rangle + \langle Au_n,v_n \rangle ) \right|\leq 2 \sqrt{\eps_n}.
\end{equation}
In particular, 
\[ \Im(\langle A v_n, u_n \rangle + \langle Au_n,v_n \rangle )  \to 0 \quad (n\to \infty).\]
\end{lemma}
\begin{proof}
First, note that by (\ref{eq:26})
  \begin{eqnarray}
&& \Im( \langle A w_n, w_n \rangle) \nonumber \\ 
&=& \eps_n \Im(\alpha) +\sqrt{\eps_n}c_n \Im(\langle Av_n,u_n \rangle + \langle Au_n,v_n \rangle) + \eps_n \Im(\langle Av_n,v_n \rangle).        \label{eq:25}
  \end{eqnarray} 
Since $\num(A) \subset \mc_+$ by assumption, the left-hand side is non-negative, so we obtain
$$ 0 \leq \eps_n \Im(\alpha) +\sqrt{\eps_n} c_n\Im(\langle Av_n,u_n \rangle + \langle Au_n,v_n \rangle) + \eps_n \Im(\langle Av_n,v_n \rangle)$$ 
and so
$$ -c_n\Im(\langle Av_n,u_n \rangle + \langle Au_n,v_n \rangle) \leq \sqrt{\eps_n} ( \Im(\alpha) + \Im(\langle Av_n,v_n \rangle)).$$
Now we do the same computations with $w_n':=u_n-\sqrt{\eps_n}c_nv_n$ and arrive at
$$ c_n\Im(\langle Av_n,u_n \rangle + \langle Au_n,v_n \rangle) \leq \sqrt{\eps_n} ( \Im(\alpha) + \Im(\langle Av_n,v_n \rangle)),$$
so taken together, and using that $|c_n|=1$,  we obtain
$$ |\Im(\langle Av_n,u_n \rangle + \langle Au_n,v_n \rangle)| \leq \sqrt{\eps_n} ( \Im(\alpha) + \Im(\langle Av_n,v_n \rangle)).$$ 
Since $\max(\|v_n\|,\|Av_n\|) \leq 1$ and $|\alpha| < 1$, an application of Cauchy-Schwarz concludes the proof of (\ref{eq:15}). The validity of (\ref{eq:16}) follows from (\ref{eq:15}),(\ref{eq:25}) and a similar application of Cauchy-Schwarz.
\end{proof}
We assumed that $0$ is a point of right-hand infinite curvature for $\partial \num(A)$, which means that for every positive null sequence $(a_n)$  we have
$K(a_n) \to \infty$ for $n \to \infty$, where
\begin{small}
 \begin{equation*}
  K(a):= \inf \left\{ 
 \frac{\Im(\langle Av,v \rangle)}{\Re^2(\langle Av,v \rangle)} : v \in \Dom(A), \|v\|=1, 0 < |\langle Av,v \rangle| < a, \Re(\langle Av,v \rangle) > 0 \right\}.
\end{equation*}
\end{small}
In order to apply this curvature assumption in our proof of Proposition \ref{prop:1}, we first need the following result.
\begin{lemma} \label{lem:pos}
For all $n \in \mn$ we have
\begin{equation}
  \Re(\langle A w_n, w_n \rangle) > 0. 
\end{equation}
\end{lemma}
\begin{proof}
From (\ref{eq:26}) we obtain 
$$  \Re(\langle Aw_n, w_n \rangle) = \eps_n \Re(\alpha)+\sqrt{\eps_n} c_n \Re(\langle Av_n,u_n \rangle + \langle Au_n,v_n \rangle) + \eps_n \Re(\langle Av_n,v_n \rangle).$$
Estimating the last term in the sum by its negative absolute value and using assumption (\ref{eq:19}) we can estimate
\begin{eqnarray}
\Re( \langle A w_n, w_n \rangle) &\geq& 
\eps_n (\Re(\alpha))/2 + \sqrt{\eps_n} c_n \Re(\langle Av_n,u_n \rangle + \langle Au_n,v_n \rangle) \nonumber \\
&=& \eps_n (\Re(\alpha))/2 + \sqrt{\eps_n} |\Re(\langle Av_n,u_n \rangle + \langle Au_n,v_n \rangle)| > 0. \label{eq:hh2}
\end{eqnarray}
For the equality we used the definition of $(c_n)$ (see (\ref{eq:13})) and in the last step we used that $\Re(\alpha) > 0$ and $\eps_n > 0$.
\end{proof}
\begin{lemma}\label{lem6} 
We have
$$ \Re(\langle Av_n,u_n \rangle + \langle Au_n,v_n \rangle) \to 0 \quad (n \to \infty).$$
\end{lemma}
\begin{proof}
From Lemma \ref{lem5} and Lemma \ref{lem4} we know that for all $n \in \mn$   
$$   \left| \Big\langle A \frac{w_n}{\|w_n\|}, \frac{w_n}{\|w_n\|} \Big\rangle \right| \leq \frac{4 \sqrt{\eps_n}}{\|w_n\|^2} \leq \frac{4\sqrt{\eps_n}}{(1-\sqrt{\eps_n})^2}  =: x_n,$$
 so by the definition of $K(x_n)$ and the fact that $\Re( \langle A w_n, w_n \rangle)>0$ by Lemma \ref{lem:pos},  we see that 
  \begin{equation}
    \label{eq:7}
    \frac{\Im\left(\Big\langle A \frac{w_n}{\|w_n\|}, \frac{w_n}{\|w_n\|} \Big\rangle\right)}{\Re^2\left(\Big\langle A \frac{w_n}{\|w_n\|}, \frac{w_n}{\|w_n\|} \Big\rangle\right)} \geq K(x_n) \qquad (n \in \mn).
  \end{equation}
Since $\|w_n\|  \leq (1+\sqrt{\eps_n})$ this implies that 
\begin{equation*}
    \frac{\Im\left(\langle A w_n, w_n \rangle\right)}{\Re^2\left(\langle A w_n, w_n \rangle\right)} \geq \frac{K(x_n)}{(1+\sqrt{\eps_n})^2} \qquad (n \in \mn)
\end{equation*} 
and so we can use Lemma \ref{lem3} to obtain
\begin{equation}
  \label{eq:12}
   \Re\left(\langle A w_n, w_n \rangle\right)\cdot \sqrt{K(x_n)} \leq 2 \sqrt{\eps_n} (1+\sqrt{\eps_n}) \qquad(n \in \mn).
\end{equation} 
From (\ref{eq:hh2}) we know that $
\Re( \langle Aw_n, w_n \rangle) \geq \sqrt{\eps_n} |\Re(\langle Av_n,u_n \rangle + \langle Au_n,v_n \rangle)|$. Plugging this into (\ref{eq:12}) we arrive at
\begin{equation*}
   |\Re(\langle Av_n,u_n \rangle + \langle Au_n,v_n \rangle)| \sqrt{ K(x_n)} \leq 2(1+\sqrt{\eps_n}) \qquad(n \in \mn).
\end{equation*}
Here the right-hand side tends to $2$ for $n \to \infty$. Moreover, since $x_n \searrow 0$ we have $K(x_n) \to \infty$. But this implies that 
$\Re(\langle Av_n,u_n \rangle + \langle Au_n,v_n \rangle) \to 0$.
\end{proof}
We are finally prepared for the proof of Proposition \ref{prop:1}: To begin, note that we can assume that 
$$ s_1 := \sup_n \|f_n\| > 0 \quad \text{and} \quad s_2 := \max(\sup_n \|Af_n\|, \sup_n \|A^*f_n\|) > 0,$$
since otherwise the implication in the proposition is trivial. Let
$$ R:=\frac{\Re(\alpha)}{2 \max(s_1,s_2)} > 0,$$
where $\alpha$ was defined in (\ref{eq:24}). Now we choose $\theta_n \in [0,2\pi)$ such that the complex numbers $ z_n=Re^{i\theta_n}\langle Af_n,u_n \rangle$ and $\mu_n=Re^{-i\theta_n}\langle Au_n,f_n \rangle$ 
have the same phase. If one of $\langle Af_n,u_n \rangle$ or $\langle Au_n,f_n \rangle$ is zero, then we choose $\theta_n$ arbitrary. With $v_n:=Re^{i\theta_n}f_n$ we then obtain
\begin{eqnarray}
&& R \left( |\langle Af_n,u_n \rangle| + |\langle Au_n, f_n \rangle| \right) 
= |z_n| + |\mu_n|=|z_n+\mu_n| \nonumber\\
&=& \sqrt{\Re^2(\langle Av_n,u_n \rangle + \langle Au_n, v_n \rangle) + \Im^2(\langle Av_n,u_n \rangle + \langle Au_n, v_n \rangle))}. \label{eq:fin}
\end{eqnarray}  
Now note that 
$$ \max(\sup_n\|v_n\|, \sup_n \|Av_n\|, \sup_n \|A^*v_n\|) \leq \Re(\alpha)/2 \leq 1$$
and
$$ \sup_n |\Re( \langle Av_n, v_n \rangle )| \leq \sup_n \|Av_n\| \|v_n\| \leq (\Re(\alpha)/2)^2 \leq \Re(\alpha)/2,$$
i.e. with this choice of $(v_n)$ the estimates (\ref{eq:8}) and (\ref{eq:19}) are satisfied. We can thus apply Lemma \ref{lem6} and \ref{lem3} to conclude that the right-hand side in (\ref{eq:fin}) tends to $0$ for $n \to \infty$. Since $R>0$ this shows that 
$$  |\langle f_n,Au_n \rangle| + |\langle f_n, A^*u_n \rangle| = |\langle Af_n,u_n \rangle| + |\langle Au_n, f_n \rangle| \to 0 \quad (n \to \infty)$$
and concludes the proof of Proposition \ref{prop:1}. 
 
\section{Non-selfadjoint Schr\"odinger operators}

Now we are going to apply our results to non-selfadjoint Schr\"odinger operators $-\Delta+V$ in $L^2(\mr^d)$. We will make the following assumptions:
\begin{enumerate}
    \item[(A1)] $V: \mr^d \to \mc$ is a locally integrable function such that the sesquilinear form 
  \begin{eqnarray*}
    \mathcal{E}(f,g) &=& \langle \nabla f, \nabla g \rangle + \int_{\mr^d} V(x) f(x) \overline{g(x)} dx, \\
\dom(\mathcal{E}) &=& H^{1}(\mr^d) \cap \{ f \in L^2(\mr^d) : V|f|^2 \in L^1(\mr^d)\},
  \end{eqnarray*}
is closed and sectorial (since $V \in L^1_{loc}$ it is also densely defined). 
\end{enumerate} 
Given this assumption, by the first representation theorem (see \cite{kato}) we can uniquely associate to $\mathcal{E}$ an $m$-sectorial operator $H=:-\Delta +V$. The numerical range of $H$ is contained in a sector $\{ \lambda : |\arg(\lambda-\gamma)| \leq \alpha \}$ for some $\gamma \in \mr$ and $\alpha \in [0,\pi/2)$ and the spectrum of $H$ is contained in the closure of its numerical range. 
\begin{enumerate}
    \item[(A2)] $\dom(H) \subset \{ f \in L^2(\mr^d) : \Im(V) f \in L^2 \}$.
\end{enumerate}
Given (A1) and (A2), we have $\dom(H) \subset \dom(H^*)$, see \cite{MR3162253}, Lemma 6. 
\begin{ex}
 For instance, using Sobolev embedding theorems it can be shown that (A1) is satisfied if $V \in L^p(\mr^d) + L^\infty(\mr^d)$, where $p=d/2$ if $d\geq 3$, $p>1$ if $d=2$ and $p=1$ if $d=1$, and (A2) is satisfied if $\Im(V) \in L^q(\mr^d)$ where $q = d$ if $d \geq 3$, $q > 2$ if $d=2$ and $q=2$ if $d=1$. However, both these conditions are not necessary for (A1) and (A2) to hold. In particular, $V$ need not decay at infinity (in a generalized sense). To mention just one such example, note that in case $d=1$ the potential $V(x)=cx^2, \Re(c) > 0$, is among the potentials satisfying (A1) and (A2) and so the non-selfadjoint harmonic oscillator $H_cf=-f'' +cx^2f$, probably the most well-studied non-selfadjoint Schr\"odinger operator (see, e.g. \cite{MR1700903, MR1911854, MR2241978}),  can also be treated by our methods.
\end{ex}

Finally, in case $d \geq 2$ we need a further assumption, which allows one to invoke a unique continuation argument.
\begin{enumerate}
    \item[(A3)] $\Re(V) \in L^p_{loc}(\mr^d)$ where $p = d/2$ if $d \geq 3$ and $p > 1$ if $d=2$.
\end{enumerate}
The following theorem was proven in \cite{MR3162253}. Here $\sigma_p(H)$ denotes the set of all eigenvalues of $H$.
\begin{thm*}
  Assume $(A1)-(A3)$. If $a+ib \in \partial \num(H) \cap \sigma_p(H)$, then for every non-empty open set $U \subset \mr^d$ the set $ \{ x \in U : \Im(V(x))=b\}$
has non-zero Lebesgue measure.
\end{thm*}
\begin{rem}
  To be precise, in \cite{MR3162253} we proved this theorem under a slightly less general assumption, namely that in (A1) we have $D(\mathcal{E})=H^1(\mr^d)$. However, a short inspection of the relevant proofs shows that this change has no influence on the validity of the results  of \cite{MR3162253}.
\end{rem}
For simplicity, instead of working with the previous theorem directly, we will concentrate on one of its corollaries.  To this end, let us introduce two further conditions. It is no exaggeration to say that at least one of them (in particular the first) is satisfied by the majority of potentials arising in applications. 
\begin{enumerate}
    \item[(C1)] There exist $x_1,x_2 \in \mr^d, x_1 \neq x_2,$ such that $\Im(V)$ is continuous at $x_1$ and $x_2$ and $\Im(V(x_1)) \neq \Im(V(x_2))$.
      \item[(C2)] $V(x) \to 0$ for $\|x\| \to \infty$.
\end{enumerate}
Note that (C1) implies that $\num(H)$ does contain an interior point. Moreover, (C2) implies that $\sigma_{ess}(H)=\sigma_{ess}(-\Delta)=[0,\infty)$.
\begin{cor}\label{cor:2}
  Assume (A1)-(A3). Then the following holds:
  \begin{enumerate}
  \item If $V$ satisfies (C1), then $\sigma_p(H) \cap \partial \num(H) = \emptyset$.
  \item If $V$ satisfies (C2), then $\sigma_p(H) \cap \partial \num(H) \subset \mr$.  
  \end{enumerate}
\end{cor}
\begin{proof}
  See \cite{MR3162253}, Corollary 6 and 7, respectively.
\end{proof}
Now we are prepared for our first theorem about points of infinite curvature of $\partial \num(H)$. In particular, it provides a necessary criterion for the closedness of $\num(H)$. 
\begin{thm}
  Assume (A1)-(A3) and (C1). If $\lambda \in \partial \num(H) \cap \num(H)$,  then the upper curvature of $\partial \num(H)$ at $\lambda$ is finite. In particular, if $\partial \num(H)$ has a point of infinite upper curvature, then $\num(H)$ is not closed.
\end{thm}
\begin{proof}
This is an immediate consequence of Theorem \ref{thm:ev} and Corollary \ref{cor:2}. 
\end{proof}
\begin{ex}\label{ex:5}
Let $W \in C_0^\infty(\mr^d, \mr_+), W \neq 0$ and consider $H=-\Delta + (1+i)W$. Then
\begin{eqnarray*}
 \num(H) &=& \{ \|\nabla f \|^2 + \langle Wf,f\rangle + i \langle Wf,f \rangle : f \in \dom(H), \|f\|=1\}\\
&\subset& \{ x + iy : x \geq y \geq 0\}.  
\end{eqnarray*}
On the other hand, we have $\sigma_{ess}(H)=[0,\infty) \subset \cnum(H)$. This shows that $0$ is a corner point of $\partial \num(H)$. In particular, $\num(H)$ is not closed.
\end{ex}
\begin{question}
The previous theorem provides a necessary criterion for the closedness of $\num(H)$. Is it possible to obtain some nice sufficient conditions as well?
\end{question}
Before stating our second result, let us remark that while generally it need not be true that $\sigma(H)$ is the union of the discrete spectrum $\sigma_{d}(H)$ (which consists of all isolated eigenvalues of finite algebraic multiplicity) and the essential spectrum of $H$, for boundary points of the spectrum we do have that 
$$ \partial \sigma(H) \subset \sigma_d(H)\: \dot\cup \:\sigma_{ess}(H).$$
This follows from the fact that for an open, connected component $U$ of $\mc \setminus \sigma_{ess}(H)$ we either have $U \subset \sigma(H)$ or $\sigma(H) \cap U \subset \sigma_d(H)$, see \cite{b_Davies}. If $\lambda \in \partial \sigma(H) \setminus \sigma_{ess}(H)$ and $U$ denotes the component of $\mc \setminus \sigma_{ess}(H)$ that contains $\lambda$, then obviously the first case cannot happen and so $\lambda \in \sigma_d(H)$.

The next theorem will show that in Example \ref{ex:5} it is no coincidence that the corner point is an element of the essential spectrum of $H$.
\begin{thm}\label{thm:5}
  Assume (A1)-(A3) and (C1). If $\lambda \in \partial \num(H)$ is a point of unilateral infinite curvature, then $\lambda \in \sigma_{ess}(H)$. In particular, if $\sigma_{ess}(H)=\emptyset$, then $\partial \num(H)$ does not have points of unilateral infinite curvature.
\end{thm}
\begin{proof}
Corollary \ref{cor:1} implies that $\lambda \in \partial \sigma(H) \cap \partial \num(H)$. Since $\sigma_d(H) \cap \partial \num(H) = \emptyset$ by Corollary \ref{cor:2}, we must have $\lambda \in \sigma_{ess}(H)$ by the discussion preceeding the theorem. 
\end{proof}
\begin{ex}
It is well known that the non-selfadjoint harmonic oscillator $H_cf = -f'' + cx^2f, \Re(c) > 0$ has compact resolvents and so $\sigma_{ess}(H_c)=\emptyset$. The previous theorem implies that  $\partial \num(H_c)$ does not have points of unilateral infinite curvature. In this case this does not come by surprise since $H_c$ is one of the few operators whose numerical range is actually known. As has been shown in \cite{MR1911854}, we have
$$ \num(H_c) = \{ t_1 +ct_2 : t_1,t_2 \geq 0, t_1t_2 \geq 1/4\}.$$
In particular, we see that in this case $\num(H_c)$ is closed.
\end{ex}
We conclude with a result about potentials satisfying (C2). Here we assume that  $\num(H)$ does contain an interior point.
\begin{thm}
    Assume (A1)-(A3) and (C2). Let $\lambda_0:=\inf \left( \sigma(H) \cap (-\infty,0] \right)$. Then the following holds:
\begin{enumerate}
    \item If $\lambda_0 \notin \partial \num(H)$, then $\partial \num(H)$ does not have points of unilateral infinite curvature.
\item If $\lambda_0 \in \partial \num(H)$, then $\lambda_0$ is the only possible point of unilateral infinite  curvature of $\partial \num(H)$.
\end{enumerate}
\end{thm}
\begin{proof}
From Theorem \ref{thm:main} and the above discussion we know that if $\lambda \in \partial \num(H)$ is a point of unilateral infinite curvature, then $\lambda \in \sigma_{ess}(H) \cup \sigma_d(H) = [0,\infty) \cup \sigma_d(H)$. But from Corollary \ref{cor:2} we know that $\sigma_p(H) \cap \partial \num(H) \subset \mr$, so $\sigma_d(H) \cap \partial \num(H) \subset (-\infty,0)$. By convexity of $\num(H)$ we thus obtain $\lambda = \lambda_0$.
\end{proof} 
\begin{rem}
  In view of the previous theorem, we see that in Example \ref{ex:5} the corner point at $0$ is actually the only point of unilateral infinite curvature of $\partial \num(H)$.
\end{rem}
\begin{question}
Above we have seen an example where $\partial \num(H)$ has no point of infinite curvature and an example where it has exactly one such point. For every positive integer $n$, is it possible to construct $H$ such that $\partial \num(H)$ has exactly $n$ points of infinite curvature?
\end{question}

\section*{Acknowledgments}
  I would like to thank the anonymous referee for his or her many helpful comments and suggestions.

\def\cydot{\leavevmode\raise.4ex\hbox{.}} \def\cprime{$'$}

\end{document}